\documentclass[12pt, a4paper, twosided, reqno]{amsart}
\usepackage{enumerate, cite}
\usepackage{hyperref}

\newtheorem{theorem}{Theorem}
\newtheorem{lemma}{Lemma}

\newtheorem{example}{Example}

\newtheorem{definition}{Definition}
\newtheorem{remark}{Remark}
\newtheorem*{thA}{Theorem A}

\newcommand{\te}{\theta}
\allowdisplaybreaks
\author[G. Pant]{Garima Pant}
\address{garima pant; department of mathematics, university of delhi, delhi-110007, india.}
\email{garimapant.m@gmail.com}
\thanks {Research work of the author is supported by research fellowship from University Grants Commission (UGC), New Delhi, India.}
\title[Growth of solutions]{Growth of solutions of second order complex linear differential equations}
\subjclass[2010]{34M10, 30D35}
\keywords {entire function, order of growth, hyper-order of growth, exponent of convergence of zeros, complex differential equation}
\begin{document}
	\maketitle
	\begin{abstract}
In this paper, we study about order and hyper-order of growth of non-trivial solutions of $f''+A(z)f'+B(z)f=0$, where $A(z)$ and $B(z)$ are entire functions having some restrictions. These restrictions involve notions of Yang's inequality, Borel exceptional value, deficient value and accumulation ray. 
	\end{abstract}
	\section{\textbf{Introduction}}
	Consider the second order linear differential equation
	\begin{equation}\label{sde}
		f''+A(z)f'+B(z)f=0
	\end{equation}
	where $A(z)$ and $B(z)(\not\equiv 0)$ are entire functions. It is a well known result by Herold that every solution of \eqref{sde} is an entire function \cite{herold}. For an entire function $f(z)$, the order of growth, the hyper-order of growth and the exponent of convergence of zeros of $f(z)$ are defined by
	$$\rho(f)=\limsup\limits_{r\to\infty}\frac{\log\log M(r,f)}{\log r},$$
	
	$$\rho_2(f)=\limsup_{r\to \infty}\frac{\log\log T(r,f)}{\log r}$$
	
	and $$\lambda(f)=\limsup\limits_{r\to\infty}\frac{\log n(r,\frac{1}{f})}{\log r}$$
	respectively, where $M(r,f)=\max_{|z|=r}|f(z)|$ is the maximum modulus of $f(z)$ on the circle of radius $r$, $T(r,f)$ is the characteristic function of $f(z)$ and  $n(r,1/f)$ is the number of zeros of $f(z)$ in $|z|\leq r$. By applying Wiman-Valiron theory it has been proved that if $A(z)$ and $B(z)$ are polynomials then, all solutions of \eqref{sde} are of finite order and vice versa. Thus, if at least one of the $A(z)$ and $B(z)$ is a transcendental entire function, then almost all solutions of \eqref{sde} are of infinite order. So it is quite interesting to ask what conditions on $A(z)$ and $B(z)$ will guarantee that all non-trivial solutions of \eqref{sde} have infinite order? There are  many results  concerning this problem. Some of them are as follows:
	
	In 1988, Gundersen \cite{finitegun} proved that if $A(z)$ and $B(z)$ are entire functions satisfying one of the following conditions:
	\begin{enumerate}
		\item $\rho(A)<\rho(B)$
		\item $A(z)$ is a polynomial and $B(z)$ is a transcendental entire function
		\item $\rho(B)<\rho(A)<1/2$
		\item $A(z)$ is a transcendental entire function with $\rho(A)=0$ and $B(z)$ is a polynomial.
	\end{enumerate}
	Then, every non-trivial solution of \eqref{sde} is of infinite order.\\
	In this sequel, Hellerstein, Miles and Rossi \cite{hmr} proved that if  $\rho(B)<\rho(A)=1/2$, then every non-trivial solutions of \eqref{sde} is of infinite order.\\
	
	For the case $\rho(B)\leq\rho(A)$ and $\rho(A)> 1/2$, it is seen that the above conclusion does not hold in general. Here, we illustrate these conditions by some examples:
	\begin{example}\rm
		The differential equation $$f''-e^{-z}f'-(e^{-z}+1)f=0$$
		has a non-trivial solution $f(z)=e^{-z}$ of order one, where $\rho(A)=\rho(B)$ and $\rho(A)>1/2$.
	\end{example}
	
	\begin{example}\rm
		The differential equation  $$f''+e^{-z}f'- f=0$$
		has $f(z)=e^{z}+1$ as its non-trivial finite order solution. Here $\rho(B)<\rho(A)$ and  $\rho(A)>1/2$. 
	\end{example}
	
	But under some conditions on $A(z)$ and $B(z)$ satisfying together with $\rho(B)<\rho(A)$, $\rho(A)> 1/2$ or $\rho(A)=\rho(B)$, it was observed that every non-trivial solution of \eqref{sde} is of infinite order. For this one may refer to  \cite{sm, sm2, jlongfab, ozawa}. \\
	In this order, some authors have used a condition that the coefficient $A(z)$ in \eqref{sde} is a non-trivial solution of the following equation
	\begin{equation}\label{polyeq}
		w''+P(z)w=0
	\end{equation}
	where $P(z)=a_mz^m+...+a_0$ is a polynomial of degree $m\geq1$, see \cite{jlongfab, wu, yz}. \\  
	
	In 2018, a problem was studied by J. Long and X. Wu which is as follows:
	\begin{thA}\rm{\cite{longwu}}
		Let $A(z)$ and $B(z)$ be two linearly independent solutions of \eqref{polyeq}. Suppose the number of accumulation rays of the zero sequence of $A(z)$ is less than $m+2$. Then, all non-trivial solutions $f$ of \eqref{sde} are of infinite order. 
	\end{thA}
	Motivated by above theorem, we consider a problem: What if the number of accumulation rays of zero sequence of a solution of \eqref{polyeq} are exactly $m+2$? Our results are based on that problem. To state and prove our main results, we first recall some definitions.
	\begin{definition}\rm{\cite{longwuwu}}
		Suppose that $f(z)$ is a meromorphic function in a finite complex plane. Let $\gamma=re^{i\te}$ be a ray from origin. For each $\epsilon>0$, the exponent of convergence of zero sequence of $f(z)$ at the ray  $\gamma=re^{i\te}$ is denoted by  $\lambda_{\te}(f)=\lim\limits_{\epsilon\to 0^+} \lambda_{\te,\epsilon}(f)$,\\
		where $$ \lambda_{\te,\epsilon}(f)=\limsup_{r\to \infty}\frac{\log n(S(\te-\epsilon,\te+\epsilon),f=0)}{\log r}$$
		here, $n(S(\te-\epsilon,\te+\epsilon),f=0)$ counts the number of zeros of $f(z)$ with multiplicities in the angular sector $S(\te-\epsilon,\te+\epsilon)=\{z : \te-\epsilon\leq\arg z\leq\te+\epsilon, |z|>0\}.$
	\end{definition}
	\begin{definition}\rm
		A ray $\gamma=re^{i\te}$ is called an accumulation ray of zero sequence of $f(z)$ if $\lambda_{\te}(f)=\rho(f).$
	\end{definition}
	
	The following remark is immediate. 
	\begin{remark}\rm
		\begin{enumerate}[(i)]
			\item The number of accumulation rays of the zero sequence of every non-trivial solution of \eqref{polyeq} is less than or equal to $m+2$.
			\item The set of accumulation rays of the zero sequence of every non-trivial solution of \eqref{polyeq} is a subset of $\{\te_j : 0\leq j\leq m+1\}$, where $\te_j=\frac{2j\pi-\arg(a_m)}{m+2} ~; 0\leq j\leq m+1.$
		\end{enumerate}
	\end{remark}
	\begin{definition} \rm{\cite{yang}}
		Suppose that $f(z)$ is a meromorphic function of order $\rho(f)\in(0,\infty)$. A ray $\arg z=\te$ from the origin is called a Borel direction of order $\rho(f)$ of $f(z)$, if for any $\epsilon>0$ and for any complex value $a\in\mathbb{C}\cup\{\infty\}$ with at most two exceptions, we have
		
		$$\limsup_{r\to\infty}\frac{\log n(S(\te - \epsilon, \te + \epsilon,r),a,f)}{\log r} = \rho(f)$$ 
		where, $n(S(\te-\epsilon,\te+\epsilon,r),a,f))$ denotes the number of zeros, counting multiplicities of $f-a$ in the region $S(\te - \epsilon,\te+\epsilon,r)=\{z: \te-\epsilon<arg z<\te+\epsilon,|z|<r\}.$
	\end{definition}
	\begin{example}\rm{\cite{yang}}
		Consider the entire function $f(z)=e^{-z^{n}}$ of order $n$. It has $2n$ Borel directions given by
		$$\arg z=\frac{(2k-1)\pi}{2n}, \qquad k=1,2,...,2n.$$
	\end{example}
	\begin{definition}\rm
		Let $f(z)$ be an entire function of finite order $\rho(f)\in(0,\infty)$. If $f(z)$ has $p$ number of Borel directions and $q$ number of finite deficient values, then $f(z)$ is called extremal for Yang's inequality if $q=p/2.$
	\end{definition}
	\begin{example}\rm
		The entire function $e^{z}$ has only one finite deficient value at $z=0$. It has two Borel directions at $\te=\pi/2$ and $\te=-\pi/2$. Thus, $e^{z}$ is an entire function extremal for Yang's inequality. 
	\end{example}
	\begin{definition}\rm{\cite{yang}}
		Suppose that $f(z)$ is a meromorphic function of finite order $\rho(f)\in(0,
		\infty)$ in the finite plane. A complex number $a$ is called an exceptional value of $f(z)$ in the sense of Borel if
		$$\limsup_{r\to\infty}\frac{\log n(r,f=a)}{\log r}<\rho(f).$$
	\end{definition}
	Now we are prepared to state our first main result in which we study about order of growth as well as hyper-order of growth of non-trivial solutions of \eqref{sde}. 
	\begin{theorem}\rm\label{main1}
		Suppose that $B(z)$ is a non-trivial solution of \eqref{polyeq} such that the number of accumulation rays of zero sequence of $B(z)$ are exactly $m+2$ and $A(z)$ satisfies any one of the following conditions:
	\begin{enumerate}
	\item $A(z)$ is an entire function extremal for Yang's inequality.
	\item  $A(z)$ is an entire function having a finite Borel exceptional value.
	\end{enumerate}
	Then, every non-trivial solution $f$ of \eqref{sde} is of infinite order.
		Moreover,
		$$\min\{\rho(A),\rho(B)\}\leq\rho_2(f)\leq\max\{\rho(A),\rho(B)\}$$
		whenever $A(z)$ and $B(z)$ are of finite order.
	\end{theorem}
	
   We illustrate the above theorem by some examples:
   
  \begin{example}\rm\label{exeyi}
   All the non-trivial solutions of the differential equation
  $$f^{''}+A(z)f^{'}+B(z)f=0,$$
  have infinite order of growth, where $B(z)$ is any non-trivial solution of $f^{''}-zf=0$, which is not a constant multiple of $\phi_{j}(z)=Ai(\alpha_{j}z)$ for some $j=1,2,3$. Here $\alpha_{j}$ is the cube root of unity and $Ai(z)$ is a special contour integral solution of $f^{''}-zf=0$, called the Airy integral and represented by
  $Ai(z)=1/2\pi\iota\int_C \exp\{(1/3)w^{3}-zw\}\,dw$, where the contour $C$ runs from $\infty$ to $0$ along $\arg w=-\pi/3$ and then $0$ to $\infty$ along $\arg w=\pi/3$. This $B(z)$ is obtained from \rm\cite{gunst}. It has exactly $1+2=3$ accumulation rays of zero sequence and $A(z)$ satisfies any one of the followings.
  \begin{enumerate}
  \item $A(z)=\int_{0}^{z}e^{-t^{2}}\,dt$ is an entire function extremal for Yang's inequality and $\rho(B)<\rho(A)$, see \rm\cite{yang}.
  \item $A(z)=e^{z^{2}}+1$ has $`1$' as a finite Borel exceptional value and $\rho(B)<\rho(A)$.
  \end{enumerate}
  \end{example}

In the final result, we consider $A(z)$ to be a transcendental entire function with a finite deficient value and $B(z)$ to be same as in Theorem \rm\ref{main1}, then we study the growth of non-trivial solutions of \eqref{sde}.

\begin{theorem}\rm\label{main3}
Suppose that $A(z)$ is a transcendental entire function with a finite deficient value and $B(z)$ is same as in Theorem \ref{main1}. Then all non-trivial solutions of equation \eqref{sde} are of infinite order.
\end{theorem}

The following example illustrates the above theorem:
\begin{example}\rm
The differential equation 
$$f^{''}+(e^{z^3}-1)f^{'}+B(z)f=0,$$
has every non-trivial solution of infinite order, where $B(z)$ is the same as in Example \rm\ref{exeyi} and $`-1$' is the finite deficient value of $A(z)=e^{z^3}-1$. 
\end{example}
	%One can also demonstrate the above main results in another aspect which is given below as remark.
	%\begin{remark}\rm
	%If $f$ is a non-trivial solution of finite order of the differential equation
		%$$f''+ e^zf'+B(z)f=0$$
		%where $B(z)$ is a non-trivial solution of \eqref{polyeq}. Then the number of accumulation rays of zero sequence of $B(z)$ must be less than $m+2$.  
	%\end{remark}
	
	\section{\textbf{Auxiliary results}}
	In this section, we state some lemmas which are used in the proofs of the main theorems. Before stating these lemmas, first we recall some elementary notions. \\
	The linear measure of a set $E\subset[0,\infty)$ is defined by 
	$m(E)=\int_E \,dt$. \\
	The logarithmic measure, lower logarithmic density and upper logarithmic density of a set $G\subset [1,\infty)$ are defined by $$m_l(G)=\int_G \frac{1}{t}\,dt$$ 
	$$\underline{\log dens}(G)=\liminf_{r\to \infty}\frac{m_l(G\cap [1,r))}{\log r},$$
	$$\overline{\log dens}(G)=\limsup_{r\to \infty}\frac{m_l(G\cap [1,r))}{\log r},$$
	respectively. Lower logarithmic density and upper logarithmic density vary between $0$ and $1$.\\
	
	The following lemma is due to Gundersen\cite{log gg} which played a pivotal role to prove many results of complex differential equations.
	\begin{lemma} \label{gunlem}
		Let $f(z)$ be a transcendental meromorphic function and let,  $k$ and $j$ be integers such that $k > j \geq 0$. Suppose that $\epsilon > 0$ and $\alpha>1$ are given real constants. Then the  following holds:
		\begin{enumerate}
			\item there exists a set $F\subset[0,2\pi) $ with $m(F) = 0$ and there exists a constant $c>0$ that depends only on $\alpha$ and integers $j$, $k$ such that if $\phi_{0}\in{[0,2\pi)\setminus F}$, then there is a constant $R_0 = R_0(\phi_0) > 1$
			such that for all $z$ satisfying $\arg z =\phi_0$ and $|z|\geq R_0$, we have
			$$\left|\frac{f^{(k)}(z)}{f^{(j)}(z)}\right|\leq c \left( \frac{T(\alpha r,f)}{r} \log^{\alpha}{r} \log{T(\alpha r,f)} \right)^{(k-j)}.$$
			If $f(z)$ is of finite order, then $f(z)$ satisfies
			$$\left|\frac{f^{(k)}(z)}{f^{(j)}(z)}\right|\leq |z|^{(k-j)(\rho(f)-1+\epsilon)}$$
			for all $z$ satisfying $\arg z =\phi_0$ and $|z| \geq R_0$.
			
			\item  there exists a set $F\subset (1,\infty)$ with finite logarithmic measure and there exists a constant $c>0$ that depends only on $\alpha$ and integers $j$, $k$ such that
			
			$$\left| \frac{f^{(k)}(z)}{f^{(j)}(z)}\right|\leq c \left( \frac{T(\alpha r,f)}{r} \log^{\alpha}{r} \log{T(\alpha r,f)} \right)^{(k-j)}.$$
			holds for all $z$ satisfying $|z|=r\notin F\cup[0,1]$.\\
			
			If $f(z)$ is of finite order, then $f(z)$ satisfies
			$$\left|\frac{f^{(k)}(z)}{f^{(j)}(z)}\right| \leq |z|^{(k-j)(\rho(f)-1+\epsilon)}$$
			
			for all $z$ satisfying $|z|\notin F\cup [0,1]$.
		\end{enumerate}
	\end{lemma}
	
	To state next lemma, first we need to recall notion of critical ray.
	\begin{definition}
		Suppose that $P(z)=a_nz^n+a_{n-1}z^{n-1}+...+a_0$ ;~$a_n\neq0$ and $\delta(P,\te)$=Re$(a_ne^{in\te} )$. A ray $\arg z=\te$ is called a critical ray of $e^{P(z)}$ if $\delta(P,\te)=0$.
	\end{definition}
	We fix some notations.
	$$E^{+}=\{\te\in[0,2\pi] : \delta(P,\te)\geq 0\};$$
	$$E^{-}=\{\te\in[0,2\pi] : \delta(P,\te)\leq 0\}.$$
	If $\phi<\psi$ such that $\psi-\phi<2\pi$, then
	$$S(\phi,\psi)=\{z\in\mathbb{C} : \phi<\arg z<\psi\};$$
	$$S(\phi,\psi,r)=\{z\in\mathbb{C} : \phi<\arg z<\psi,|z|<r\}.$$
	
	Critical rays of $e^{P(z)}$ divide the whole complex plane into $2n$ sectors of equal length $\pi/n$. Suppose that $\phi_i$ and $ \psi_i$ ( $1\le i\leq n$) are critical rays of $e^{P(z)}$ such that $0\leq\phi_1<\psi_1<\phi_2<\psi_2<...<\phi_n<\psi_n$ and $\phi_{n+1}=2\pi+\phi_1$. These critical rays form $2n$ disjoint sectors $S(\phi_i,\psi_i)$ and $S(\psi_i,\phi_{i+1})$ ; $1\leq i\leq n$ in which $e^{P(z)}$ satisfies $\delta(P,\te)>0$ and $\delta(P,\te)<0$, respectively.
	\begin{example}
		The function $e^{-z}$ has two critical rays at $\te=\pi/2$ and $\te=-\pi/2$. Also $E^{+}=[\pi/2,3\pi/2]$ and $E^{-}=[-\pi/2,\pi/2]$.
	\end{example}
	Now we are prepared to state next lemma which gives estimate for an entire function with integral order of growth.
	\begin{lemma}\rm{\cite{banklang}}\label{polyn}
		Let $A(z)=d(z)e^{Q(z)}$ be an entire function, where $Q(z)$ is a polynomial of degree $n\geq 1$, and $d(z)$ is an analytic function such that $\rho(d)<\rho(A)=$ deg $Q(z)$. Then for given $\epsilon>0$, there exists a set $E\subset[0,2\pi)$ with linear measure zero, such that
		\begin{enumerate}
			\item if $\te\in E^{+}\setminus E$, there exists a $R(\te)>1$ such that 
			$$|A(re^{i\te})|\geq \exp((1-\epsilon)\delta(Q,\te)r^n)$$
			holds for all $r>R(\te)$.
			\item if $\te\in E^{-}\setminus E$, there exists a $R(\te)>1$ such that
			$$|A(re^{i\te})|\leq \exp((1-\epsilon)\delta(Q,\te)r^n)$$
			holds for all $r>R(\te)$.
		\end{enumerate}
	\end{lemma}
	\begin{remark}
		From the above lemma, we mention a clear notion of $E^+$ and $E^-$ which will be used in the proof of main theorems.\\
		$$E^{+}=\bigcup\limits_{i=1}^{i=n}(\phi_i,\psi_i), \qquad \qquad E^{-}=\bigcup\limits_{i=1}^{i=n}(\psi_i,\phi_{i+1}).$$
	\end{remark}
	
	To state next lemma, we first fix some notations.\\
	Let $0\leq\alpha<\beta<2\pi$ and $S(\alpha,\beta)=\{z : \alpha<\arg z<\beta\}$ be a sector. $\overline{S}$ denotes closure of $S$. Suppose that $f(z)$ is an entire function of finite order $\rho(f)\in(0,\infty)$. We say that $f(z)$ blows up exponentially in $\overline{S}$ if
	$$\lim_{r\to\infty}\frac{\log\log
		|f(re^{i\te})|}{\log r}=\rho(f)$$
	holds for any $\te\in(\alpha,\beta)$.  We also say that $f(z)$ decays to zero exponentially in $\overline{S}$ if
	$$\lim_{r\to\infty}\frac{\log\log
		|f(re^{i\te})|^{-1}}{\log r}=\rho(f)$$
	holds for any $\te\in(\alpha,\beta)$.\\
	
	Now we state next lemma which was originally given by Hille \rm{\cite{hille}}, one can also find in \cite{shin, jlongfab}. This lemma plays an important role to prove our results. 
	\begin{lemma}\label{hillelemma}
		Suppose that $w$ is a non-trivial solution of \eqref{polyeq}. Set  $S_j=S(\te_j,\te_{j+1})$, where $\te_j=\frac{2\pi j-\arg (a_m)}{m+2} ~; 0\leq j\leq m+1$. Then $w$ satisfies the following properties:
		\begin{enumerate}
			\item In each sector $S_j$, $w$ either blows up or decays to zero exponentially.
			\item If $w$ decays to zero in $S_j$, for some $j$, then it must blow up in $S_{j-1}$ and $S_{j+1}$. However, it is possible for $w$ to blow up in many adjacent sectors.
			\item If $w$ decays to zero in $S_j$, then $w$ has at most finitely many zeros in any
			closed sub-sector within $S_{j-1}\cup\overline{S_j}\cup S_{j+1}$.
			\item If $w$ blows up in $S_{j-1}$ and $S_j$, then for each $\epsilon>0$, $w$ has infinitely many zeros in each sector $\overline{S}(\te_j-\epsilon,\te_j+\epsilon)$.
		\end{enumerate}
	\end{lemma}
The following lemma gives estimate for a meromorphic function.
\begin{lemma}\rm\cite{kwonkim}\label{kwonkimlemma}
Suppose that $f(z)$ is a meromorphic function of finite order $\rho$. Then for the given $\delta>0$ and $0<l<1/2$, there exists a constant $\kappa(\rho,\delta)$ and a set $E_{\delta}\subset[0,\infty)$ of lower logarithmic density greater than $1-\delta$ such that for all $r\in E_{\delta}$ and for every interval $I$ of length $l$, we have
$$r\int_{I}\left|\frac{f{'}(re^{\iota\theta})}{f(re^{\iota\theta})}\right|\,d\theta<\kappa(\rho,\delta)(l\log\frac{1}{l})T(r,f).$$ 
\end{lemma}
	
	Next lemma gives an upper bound for hyper-order of growth of every solution $f$ of \eqref{sde}.
	
	\begin{lemma}\rm{\cite{cz}}\label{hyperless}
		Suppose that A(z) and B(z) are entire functions of finite order. Then each solution $f$ of \eqref{sde} satisfies
		$$\limsup_{r\to \infty}\frac{\log\log T(r,f)}{\log r}\leq\max\{ \rho(A), \rho(B)\}.$$
	\end{lemma}
	
	The following lemma provides lower bound for hyper-order of growth of non-trivial solution $f$ of \eqref{sde}.
	\begin{lemma}\rm{\cite{kwon}}\label{hypergreater}
		Suppose that $A(z)$ and $B(z)$ are entire functions with $\rho(A)<\rho(B)$ or $\rho(B)<\rho(A)<1/2$. Then every non-trivial solution $f$ of \eqref{sde} satisfies
		$$\limsup_{r\to \infty}\frac{\log\log T(r,f)}{\log r}\geq\max\{\rho(A),\rho(B)\}.$$	
	\end{lemma}
	
	\section{\textbf{Proof of theorems}}
	
	Before giving the proof of main results, we first state and prove an auxiliary result which would be required to prove our main theorems. 
	
	\begin{lemma}\label{new}
		Let $w$ be a non-trivial solution of \eqref{polyeq} such that the number of accumulation rays of zeros sequence of $w(z)$  are exactly $m+2$. Then, $w(z)$ blows up exponentially in each sector $S_j=S(\te_j,\te_{j+1})$, where $\te_j=\frac{2\pi j-\arg (a_m)}{m+2} ~; 0\leq j\leq m+1$.
	\end{lemma}

\begin{proof}
 Suppose that there exist a sector $S_j$ in which $w(z)$ decays to zero exponentially. Then using Lemma [\ref{hillelemma}], $w(z)$ has at most finitely many zeros in any closed sub-sector within $S_{j-1}\cup\overline{S_j}\cup S_{j+1}$. But for the ray $\arg z=\te_j$ we have, $\lambda_{\te_j}(f)=\rho(f)$. This implies that there are infinite number of zeros clustering around the ray $\arg z=\te_j$. Which is a contradiction. Hence, $w(z)$ blows up exponentially in each sector $S_j.$\\
\end{proof}	
	
 \begin{proof}[\underline{Proof of Theorem \rm{\ref{main1}}}]
 Let $f$ be a non-trivial solution of finite order which is contrary to the assertion. We target to prove theorem by contradiction.
		
\begin{enumerate}
\item	Suppose that $A(z)$ is an entire function extremal for Yang's inequality and let, $A(z)$ has $q$ finite deficient values say,  $b_1,b_2,...,b_q$. Then, $A(z)$ has $2q$ Borel directions say, $\phi_1,\phi_2,...,\phi_{2q}$ which divides whole complex plane into $2q$ sectors say, $\Omega_j(\phi_j,\phi_{j+1})$ where $1\leq j\leq 2q$  and $\phi_{2q+1}=\phi_1+2\pi$.\\
As $A(z)$ is an extremal for Yang's inequality, so for the alternative sectors say, $\Omega_1,\Omega_3,...,\Omega_{2q-1}$, there exists $\phi\in(\phi_j,\phi_{j+1})$ ;  $j=1,3,...,2q-1$, such that $A(z)$ satisfies
		
$$\limsup_{r\to\infty}\frac{\log\log|A(re^{i\phi})|}{\log r}=\rho(A)$$
		
and for the remaining sectors $\Omega_j$, for every deficient value $b_j$, where $j=1,2,...,q$, there exists a corresponding sector domain $\Omega_j ; ~j=2,4,...,2q$ such that
$$\log\frac{1}{|A(z)-b_j|}>C(\phi_{j},\phi_{j+1},\epsilon,\delta(b_j,A))T(r,A)$$
holds for $z\in \Omega(\phi_{j}+\epsilon,\phi_{j+1}-\epsilon,r_0,\infty)$, where $C(\phi_{j},\phi_{j+1},\epsilon,\delta(b_j,A))$
is a constant depending on $\phi_{j},\phi_{j+1},\epsilon$ and $\delta(b_j,A)$. For simplicity we can denote it by $C$.\\
Without loss of generality, corresponding to a finite deficient value $b_{j_0}$  we can take a sector $\Omega_{2i} ~;1\leq i\leq q$ such that 
		\begin{equation}\label{yeq}
			\log\frac{1}{|A(z)-b_{j_0}|}>C T(r,A)
		\end{equation}
		
		holds for $z\in \Omega(\phi_{2i}+\epsilon,\phi_{2i+1}-\epsilon,r_0,\infty)$\\
		
		Using Lemma [\ref{new}], $B(z)$ blows up exponentially in each sector $S_j$ ; $j=0,1,...,m+1$. Therefore, there exist a sector $S_k(\te_k,\te_{k+1})$ such that  $B(z)$ blows up exponentially for any $\te\in (\te_k,\te_{k+1})\cap(\phi_{2i}+\epsilon,\phi_{2i+1}-\epsilon)$ for some $1\leq i\leq q$ and we have
		\begin{equation}\label{bloweq}
			\lim_{r\to\infty}\frac{\log\log
				|B(re^{i\te})|}{\log r}=\rho(B)
		\end{equation}
		
		for all sufficiently large $r$.\\
		
		From Lemma [\ref{gunlem}], there exists a set
		$F\subset[0,2\pi) $ with $m(F) = 0$ such that if $\te_{0}\in[0,2\pi)\setminus F$, then there is a constant $R_0 = R_0(\te_0) > 1$
		such that for all $z$ satisfying $\arg z =\te_0$ and $|z|\geq R_0$, we have
		\begin{equation}\label{guneq}
			\left| \frac{f^{(k)}(z)}{f(z)}\right| \leq |z|^{2\rho(f)} , \qquad k=1,2. 	
		\end{equation}
		
		Combining $\eqref{sde},\eqref{yeq},\eqref{bloweq}$ and $\eqref{guneq}$, there exists a sequence $z=re^{i\te}$ such that $\te\in(\te_k,\te_{k+1})\cap(\phi_{2i}+\epsilon,\phi_{2i+1}-\epsilon)\setminus F$, for some $1\leq i\leq q$ and $r>\max\{r_0,R_0\}$,\\ we have
		\begin{align*}
			|B(z)|&\leq \left|\frac{f''(z)}{f(z)}\right|+|A(z)|\left|\frac{f'(z)}{f(z)}\right| \\
			\exp(r^{\rho(B)-\epsilon'})&\leq\left|\frac{f''(re^{i\te})}{f(re^{i\te})}\right|+|(A(re^{i\te})-b_{j_0})+b_{j_0}|\left|\frac{f'(re^{i\te})}{f(re^{i\te})}\right|\\
			&<r^{2\rho(f)}(1+\exp(-CT(r,A))+|b_{j_0}|)
		\end{align*}
		which is a contradiction for sufficiently large $r$. Therefore, every non-trivial solution $f$ of \eqref{sde} is of infinite order.\\
		
		Now to show that  
		$$\min\{\rho(A),\rho(B)\}\leq\rho_2(f)\leq\max\{\rho(A),\rho(B)\},$$
		we need to investigate the following cases:
		\begin{enumerate}[(i)]
			\item If $\rho(A)<\rho(B)$, Then using Lemma [\ref{hyperless}] and Lemma[\ref{hypergreater}], we obtain $\rho_2(f)=\rho(B)$.
			
			\item If $\rho(B)\leq\rho(A)$, by applying Lemma [\ref{gunlem}], there exists a set $F\subset [0,2\pi)$ with $m(F)=0$ and a constant $c>0$ such that if $\te_{0}\in [0,2\pi)\setminus F$, then there is a constant $R_0(\te_0)>1$ such that
			\begin{equation}\label{guninfeq}
				\left| \frac{f^{(k)}(z)}{f^{(j)}(z)}\right| \leq c  \left[T(2r,f)\right]^{2(k-j)}
			\end{equation}
			holds for all $z$ satisfying $|z|\geq R_0$ and $\arg z=\te_0$.\\
			
			Combining  $\eqref{sde},
			\eqref{yeq}, \eqref{bloweq}$ and $\eqref{guninfeq}$, there exists a sequence $z=re^{i\te}$ such that for $\te\in(\te_k,\te_{k+1})\cap(\phi_{2i}+\epsilon,\phi_{2i+1}-\epsilon)\setminus F$ for some $1\leq i\leq q$ and $r>\max\{R_0,r_0\}$ we have
			\begin{align*}
				\exp(r^{\rho(B)-\epsilon'})&< |B(re^{\iota \theta})|\leq \left| \frac{f''(re^{\iota \theta})}{f(re^{\iota \theta})}\right| +|(A(re^{\iota \theta})-b_{j_{0}})+b_{j_{0}}| \left| \frac{f'(re^{\iota \theta})}{f(re^{\iota \theta})}\right| \\
				& \leq c T(2r,f)^4
				+(\exp(-CT(r,A))+|b_{j_0}|)cT(2r,f)^2 \\
				&\leq cT(2r,f)^4(1+o(1)).
			\end{align*}
			Hence, we get
			\begin{equation}\label{hypereq}
				\rho(B)-\epsilon'\leq \limsup_{r\to \infty} \frac{\log \log T(r,f)}{\log r}.
			\end{equation}
			where $\epsilon'>0$ is arbitrary.\\
			From \eqref{hypereq} and Lemma [\ref{hyperless}]  , we obtain
			$$\rho(B)\leq\rho_2(f)\leq\rho(A).$$ 
		\end{enumerate}

	\item Suppose that $a$ is a Borel exceptional value of $A(z)$. Then, $A(z)-a$ has zero as a Borel exceptional value. Applying Weierstrass factorisation theorem, we have
		$$A(z)-a=g(z)=d(z)e^{Q(z)},$$
		where, $Q(z)=b_nz^n+...+b_0 ;~b_n\neq0$ and $\rho(d)<\rho(A)=$ deg $Q(z)$. \\
		This implies $$|A(z)-a|=|d(z)e^{Q(z)}|=|d(z)|e^{Re\{Q(z)\}}$$
		Applying Lemma [\ref{polyn}], for $\te\in E^-\setminus E$, there exist a $R(\te)>1$ such that
		\begin{equation}\label{blleq}
			|A(re^{i\te})-a|<\exp((1-\epsilon)\delta(Q,\te)r^n)
		\end{equation}
		holds for all $r>R(\te)$. For simplicity, we can say that \eqref{blleq} holds for $\te\in\bigcup\limits_{i=1}^{n}(\psi_i,\phi_{i+1})\setminus E$ and $r>R(\te)$.\\
		
		Using Lemma [\ref{new}], there exists a sector $S_k(\te_k,\te_{k+1})$ such that  $B(z)$ blows up exponentially for any $\te\in(\psi_i,\phi_{i+1})\cap(\te_k,\te_{k+1})\setminus E$, for some $1\leq i\leq n$ and we have
		
		\begin{equation}\label{bloweqq}
			\lim_{r\to\infty}\frac{\log\log
				|B(re^{i\te})|}{\log r}=\rho(B)
		\end{equation}
		for all $r>R$.\\

		Using Lemma [\ref{gunlem}], there exists a set $F\subset (1,\infty)$ with finite logarithmic measure such that
		\begin{equation}\label{guneqq}
			\left| \frac{f^{(k)}(z)}{f(z)}\right| \leq |z|^{2\rho(f)} , \qquad k=1,2 	
		\end{equation}
		holds for all $z$ satisfying $|z|\notin F\cup [0,1]$.\\
		
		All together with $\eqref{sde}, \eqref{blleq}$, $\eqref{bloweqq}$, and $\eqref{guneqq}$, there exists a sequence $z=re^{i\te}$ such that for $\te\in(\psi_i,\phi_{i+1})\cap(\te_k,\te_{k+1})\setminus E$, for some $1\leq i\leq n$ and for all sufficient large $r\notin F$, we have
		\begin{align*}
			\exp(r^{\rho(B)-\epsilon})&\leq\left|\frac{f''(re^{i\te})}{f(re^{i\te})}\right|+|(A(re^{i\te})-a)+a|\left|\frac{f'(re^{i\te})}{f(re^{i\te})}\right|\\
			&\leq r^{2\rho(f)}+(\exp((1-\epsilon)\delta(Q,\te)r^n)+|a|)r^{2\rho(f)}\\
			&<r^{2\rho(f)}(1+o(1))
		\end{align*}
		This is a contradiction for sufficiently large $r$. Thus, every non-trivial solution $f$ of \eqref{sde} is of infinite order.\\
		Proceeding on similar lines as in $(1)$, we obtain the range of hyper-order of non-trivial solutions $f$ of \eqref{sde}. We omit the details. This completes the proof.
	\end{enumerate}	
	\end{proof}

\begin{proof}[\underline{Proof of Theorem \rm{\ref{main3}}}]
Let $f$ be a finite order non-trivial solution of equation \eqref{sde}. we aim to prove theorem by contradiction.
Suppose $c\in\mathbb{C}$ is a finite deficient value of $A(z)$. Then it follows from the definition that
$$\liminf_{r\to \infty}\frac{m(r,\frac{1}{A(z)-c})}{T(r,A)}=2\alpha>0$$
This gives
$$m(r,\frac{1}{A(z)-c})\geq\alpha T(r,A)$$
for all sufficiently large $r$.
Thus, for sufficiently large $r$, there exists $z_{r}=re^{\iota\theta_{r}}$ such that
\begin{equation*}
\log|A(z_r)-c|\leq -\alpha T(r,A).
\end{equation*}
From Lemma [\ref{kwonkimlemma}], we choose $\delta>0$ and $0<l<1/2$ in such a way that $\kappa(\rho(A),\delta)(l\log(1/l))$ is sufficiently small. We can also choose $\phi>0$,  $|\theta_{r}-\phi|\leq l$ such that
\begin{align*}
\log|A(re^{\iota \theta})-c|&=\log|A(re^{\iota \theta_{r}})-c|+\int_{\theta_{r}}^{\theta}\frac{d}{dt}\log|A(re^{it})-c|\, dt\\
&\leq -\alpha T(r,A)+r\int_{\theta_{r}}^{\theta}\left|\frac{(A-c)^{'}(re^{\iota t})}{(A-c)(re^{\iota t})}\right|\, dt\\
&\leq -\alpha T(r,A)+\kappa(\rho(A),\delta)(l\log(1/l))T(r,A)\\
&\leq 0
\end{align*}
holds for all $\theta\in[\theta_{r}-\phi,\theta_{r}+\phi]$ and for all sufficiently large $r\in E_{\delta}$, where $\underline{\log dens}(E_{\delta})>1-\delta$.
Thus we have
\begin{equation}\label{defeq}
A(re^{\iota \theta})\leq 1+c.
\end{equation}
for all sufficiently large $r\in E_{\delta}$ and for all $\theta\in[\theta_{r}-\phi,\theta_{r}+\phi]$.\\
	Using Lemma [\ref{new}], there exists a sector $S_k(\te_k,\te_{k+1})$ such that  $B(z)$ blows up exponentially for any $\te\in[\theta_{r}-\phi,\theta_{r}+\phi]\cap(\te_k,\te_{k+1})$ and we have

\begin{equation}\label{defblowupeq}
	\lim_{r\to\infty}\frac{\log\log
		|B(re^{i\te})|}{\log r}=\rho(B)
\end{equation}
for all sufficiently large $r\in E_{\delta}$.\\
Combining equations $\eqref{sde},\eqref{guneqq},\eqref{defeq}$ and $\eqref{defblowupeq}$, there exists a sequence $z=re^{\iota\theta}$ such that for all sufficiently large $r\in E_{\delta}\setminus (F\cup[0,1])$ and for all  $\theta\in[\theta_{r}-\phi,\theta_{r}+\phi]\cap(\te_k,\te_{k+1})$, we have
\begin{align*}
|B(re^{\iota\theta})|&\leq \left|\frac{f''(re^{\iota\theta})}{f(re^{\iota\theta})}\right|+|A(re^{\iota\theta})|\left|\frac{f'(re^{\iota\theta})}{f(re^{\iota\theta})}\right| \\
\exp(r^{\rho(B)-\epsilon'})&<\left|\frac{f''(re^{\iota\te})}{f(re^{\iota\te})}\right|+|(A(re^{\iota\te})|\left|\frac{f'(re^{\iota\te})}{f(re^{\iota\te})}\right|\\
&\leq r^{2\rho(f)}(2+c),
\end{align*}
which is a contradiction for sufficiently large $r\in E_{\delta}$. Hence, all non-trivial solutions of \eqref{sde} have infinite order.
\end{proof}

\end{document}